\documentclass{amsart}
\usepackage{amsmath,amssymb,amsfonts,amsthm}
\usepackage[dvips]{graphicx}
\usepackage{setspace}
\usepackage{url}
\usepackage{hyperref}
\usepackage{enumerate}

\begin{document}
\newcommand{\sidebar}[1]{\vskip10pt\noindent
 \hskip.70truein\vrule width2.0pt\hskip.5em
 \vbox{\hsize= 4truein\noindent\footnotesize\relax #1 }\vskip10pt\noindent}

\newcommand{\omegaone}{\ensuremath{\omega_1}}
\newcommand{\lomegaone}{\ensuremath{\mathcal{L}_{\omega_1,\omega}}}
\newcommand{\alephs}[1]{\ensuremath{\aleph_{#1}}}
\newcommand{\alephalpha}{\alephs{\alpha}}
\newcommand{\alephomegaone}{\alephs{\omegaone}}
\newcommand{\alephalphaplus}{\alephs{\alpha+1}}
\newcommand{\beths}[1]{\ensuremath{\beth_{#1}}}
\newcommand{\bethalpha}{\beths{\alpha}}
\newcommand{\bethomegaone}{\beths{\omegaone}}
\newcommand{\bethalphaplus}{\beths{\alpha+1}}
\newcommand{\z}{\ensuremath{\mathcal{Z}}}
\newcommand{\M}{\ensuremath{\mathcal{M}}}
\newcommand{\N}{\ensuremath{\mathcal{N}}}
\newcommand{\A}{\ensuremath{\mathcal{A}}}
\newcommand{\B}{\ensuremath{\mathcal{B}}}
\newcommand{\F}{\ensuremath{\mathcal{F}}}
\newcommand{\D}{\ensuremath{\mathcal{D}}}
\newcommand{\C}{\ensuremath{\mathcal{C}}}
\newcommand{\E}{\ensuremath{\mathcal{E}}}
\newcommand{\G}{\ensuremath{\mathcal{G}}}
\renewcommand{\P}{\mathbb{P}}
\renewcommand{\C}{\mathbb{C}}
\newcommand{\e}{\ensuremath{\varepsilon}}
\newcommand{\noe}{\ensuremath{\e\kern-0.5em{\slash}}}
\newcommand{\apspec}{AP-\ensuremath{Spec(\phi)} }
\newcommand{\JEPspec}{JEP-\ensuremath{Spec(\phi)} }
\newcommand{\mmspec}{MM-\ensuremath{Spec(\phi)} }
\newcommand{\apspecpsi}{AP-\ensuremath{Spec(\psi)} }
\newcommand{\JEPspecpsi}{JEP-\ensuremath{Spec(\psi)} }
\newcommand{\mmspecpsi}{MM-\ensuremath{Spec(\psi)} }
\newcommand{\specpsi}{\ensuremath{Spec(\psi)} }

\newcommand{\lang}[1]{\ensuremath{\mathcal{L}_{#1}}}
\newcommand{\langhat}{\ensuremath{\widehat{\lang{}}}}
\newcommand{\kappaplus}{\ensuremath{\kappa^{+}}}
\newcommand{\kappaplusplus}{\ensuremath{\kappa^{++}}}
\newcommand{\continuum}{\ensuremath{2^{\aleph_0}}}
\newcommand{\twoalephone}{\ensuremath{2^{\aleph_1}}}
\newcommand{\dom}{\mathop{\mathrm{dom}}\nolimits}
\newcommand{\ran}{\mathop{\mathrm{ran}}\nolimits}
\newcommand{\baleph}{\ensuremath{\B(\aleph_1)}}
\newcommand{\one}{\ensuremath{\mathbf{1}}}
\newcommand{\zero}{\ensuremath{\mathbf{0}}}
\newcommand{\bethtwo}{2^{2^\kappa}}

\def\K{\mbox{\boldmath $K$}}
\def\subm{\prec_{\mbox{\scriptsize \boldmath $K$}}}

\newtheorem{theorem}{Theorem}[section]
\newtheorem{lemma}[theorem]{Lemma}
\newtheorem{corollary}[theorem]{Corollary}
\newtheorem{proposition}[theorem]{Proposition}
\theoremstyle{definition}
\newtheorem{definition}[theorem]{Definition}
\newtheorem{claim}[theorem]{Claim}
\newtheorem{subclaim}{Subclaim}\newtheorem{note}{Note}
\newtheorem{observation}[theorem]{Observation}
\newtheorem{openq}[theorem]{Open Question}
\newtheorem{openqs}[theorem]{Open Questions}
\newtheorem{fact}[theorem]{Fact}
\newtheorem{example}[theorem]{Example}
\newtheorem{xca}[theorem]{Exercise}
\theoremstyle{remark}
\newtheorem{remark}[theorem]{Remark}
\newtheorem{convention}[theorem]{Convention}
\numberwithin{equation}{section}
\newtheorem{notation}[theorem]{Notation}
\newtheorem{conjecture}[theorem]{Conjecture}

\newcommand{\cP}{\mathcal{P}}
\newcommand{\cL}{\mathcal{L}}
\newcommand{\comment}[1]{}
\newcommand{\LS}{LS}

\title{A Lower Bound for the Hanf Number for Joint Embedding}

\author{ Will Boney}
\address{Department of Mathematics, Harvard University, Cambridge, MA, USA}
\email{wboney@math.harvard.edu}

\author {Ioannis  Souldatos}
\address{Aristotle University of Thessaloniki, Thessaloniki, 54124, Greece }
\email{souldatos@math.auth.gr}

 

\date{\today}

\begin{abstract}
In \cite{BaldwinBoneyHanfNumbers} the authors show that if $\mu$ is a strongly compact cardinal, $\K$ is an Abstract Elementary 
Class (AEC) with $\LS(\K)<\mu$, and $\K$ satisfies joint embedding (amalgamation) cofinally below $\mu$, 
then $\K$ satisfies joint embedding (amalgamation) in all cardinals $\ge \mu$. The question was raised if the strongly compact 
upper bound was optimal. 

In this paper we prove the existence of an AEC $\K$ that can be axiomatized by an $\lomegaone$-sentence in a countable 
vocabulary, so that if $\mu$ is the first measurable cardinal, then 
\begin{enumerate}
 \item $\K$ satisfies joint embedding cofinally below $\mu$ ;
 \item $\K$ fails joint embedding cofinally below $\mu$; and
  \item $\K$ satisfies joint embedding above  $\mu$.
\end{enumerate}
Moreover, the  example can be generalized to an AEC $\K^\chi$ axiomatized in $\cL_{\chi^+, \omega}$, in a vocabulary of 
size $\chi$, such that (1)-(3) hold with $\mu$ being the first measurable above $\chi$.

This proves that the Hanf number for joint embedding is contained in the interval between 
the first measurable and the first strongly compact. Since these two cardinals can consistently coincide, the 
upper bound from \cite{BaldwinBoneyHanfNumbers} is consistently optimal.

This is also the first example of a sentence whose joint embedding spectrum is (consistently) neither an initial nor an  
eventual interval of cardinals. By Theorem \ref{club}, for any club $C$ on the first measurable $\mu$, it is consistent that JEP 
holds exactly on $\lim C$ and everywhere above $\mu$. 
\end{abstract}
 \maketitle

\section{Background}

In \cite[Conjecture 9.3]{GrossbergClassicationAECs}, Grossberg made (essentially) the following conjecture.
\begin{conjecture} For every $\kappa$, there is a cardinal $\mu(\kappa)$ such that for every Abstract Elementary Class (AEC) 
$\K$, if $\K$ has the $\mu(\LS(\K))$-amalgamation property\footnote{Throughout the paper we assume the common convention that 
$\LS(\K) \geq |\tau(\K)|$.}  then $\K$ has the $\lambda$-amalgamation property for all $\lambda\ge \mu(\LS(\K))$.  
\end{conjecture}

This cardinal $\mu(\LS(\K))$ (if it exists) is called the Hanf number for amalgamation, and we can define similarly what is means 
for 
a cardinal to be the Hanf number for joint embedding. 

Baldwin and Boney in \cite{BaldwinBoneyHanfNumbers} proved the existence of a Hanf number for joint embedding (and amalgamation 
and a few other variants of these two properties) from large cardinals, although their definition of `Hanf number' is slightly 
different from Grossberg's:  if $\mu$ is a strongly compact cardinal, $\K$ is an AEC with $\LS(\K)<\mu$, and $\K$ satisfies 
amalgamation cofinally below $\mu$, then $\K$ satisfies amalgamation in all cardinals $\ge \mu$. 

The confusion comes from the fact that Hanf number is most often used for the property of model existence.  In this case, there 
are a number of equivalent formulations.
\begin{fact}\label{hanfme} 
Let $(\K,\subm)$ be an AEC, let $P(\lambda)$ be the property that ``there exists a model of size $\lambda$ in $\K$''(model-existence), and let $\kappa$ be the Hanf number of $\K$. 
The following properties are all equivalent:
\begin{enumerate}
 \item $P(\lambda)$ holds in \emph{some}  $\lambda\ge\kappa$;
 \item $P(\lambda)$ holds in \emph{every} $\lambda\ge\kappa$;
 \item $P(\lambda)$ holds for \emph{cofinally} many $\lambda<\kappa$;
 \item $P(\lambda)$ holds for \emph{eventually} many $\lambda<\kappa$; and
 \item $P(\lambda)$ holds in \emph{every} cardinality.
\end{enumerate}

\end{fact}

These equivalences heavily use the downward-closed nature of model existence (and the computation of the Hanf number using 
Morley's Omitting Types Theorem for (3) and (4)). However this is not the case for other properties, such as joint embedding and 
amalgamation, and the 
equivalence of (1)-(5) fails in these cases.

More importantly, the computation of the Hanf number does not use any of the specifics of the AEC $\K$ except for the $LS(\K)$. So, two AECs with the same $LS(\K)$ will yield the same Hanf number. 
This motivates Definition \ref{hanfp}, the Hanf number for a general property $P$. 
\\

\emph{For the purposes of this paper we take the definition of the Hanf number to be that of Baldwin-Boney 
(cf. \cite{BaldwinBoneyHanfNumbers}).} 

\begin{definition}\label{hanfp} 
Fix a property $P(\lambda)$.  A function $\mu_P(\kappa)$ from cardinals to cardinals will be called the \emph{Hanf number for $P$} iff it 
satisfies the following:

\begin{enumerate}
 \item If $\K$ is an AEC that satisfies property $P(\lambda)$ for cofinally many $\lambda<\mu_P(\LS(\K))$, then $\K$ 
satisfies $P(\lambda)$ in every $\lambda\ge\mu_P(\LS(\K))$ and
\item $\mu_P(\LS(\K))$ is the least such cardinal. 
\end{enumerate}

If $P$ is obvious from the context, then we will write $\mu(\kappa)$ instead of $\mu_P(\kappa)$.
\end{definition}

In the language of Fact \ref{hanfme}, Definition \ref{hanfp} defines the Hanf number for property $P$ to be the least cardinal such that the implication $(3)\rightarrow (2)$  holds true for all AECs with the same LS-number. 
Because the Hanf number is a function, we investigate certain values of this function. We fix property $P$ and $LS(\K)$  and we ask the value of $\mu_P(LS(\K))$. 
The example stated in the abstract proves that $\mu_{JEP}(\chi)$ cannot be smaller than the first measurable cardinal above $\chi$. 

Properties $P$ of interest are joint embedding, amalgamation, categoricity, existence of maximal models, etc. In this paper, 
we focus on the properties of joint embedding (JEP) and amalgamation (AP), especially the former.

In a parallel line of inquiry we ask the following: Fix $LS(\K)$ and let $P$ be the property of joint embedding, or amalgamation, or some other property of interest. For which of the implications between $(1)$-$(5)$, other than $(3)\rightarrow (2)$,  there exists some $\kappa$ that makes the implication true? What is the least such $\kappa$? Provide counterexamples when no such $\kappa$ exists. 

One of the contributions of this paper is to provide an example that (3) does not imply (4) when $P$ is JEP and $\kappa$ is the 
first measurable\footnote{In this example $LS(\K)=\aleph_0$, but the construction works for any $LS(\K)$  smaller than the first measurable cardinal.}. In 
particular, as stated in the abstract, there exists an AEC $\K$ that both satisfies and fails JEP cofinally below the first measurable.  

This example also relates to Question 4.0.2 in \cite{BKLdap}, that asks whether there is an AEC, in particular one 
defined by an $\lomegaone$-sentence, whose finite amalgamation spectrum is not an interval (that is, amalgamation restricted to 
the $\aleph_n$ for $n<\omega$).  If we drop the restriction to a finite spectrum and replace AP by JEP, we prove that the answer 
is positive. Notice that while our example exhibits quite interesting JEP-spectrum, the same is not true for the AP-spectrum. By 
Theorem \ref{negative}, $\K^\chi$ fails amalgamation in every cardinal above $2^{\chi^+}$. This leaves open
the corresponding question (does (3) imply (4)), when $P$ is the amalgamation property and $\kappa$ is the first measurable cardinal. 

If $P$ is either JEP or AP, and $\kappa$ the first measurable or the first strongly compact cardinal, the status of the implications $(1)\rightarrow (2)$ (as envisioned by Grossberg) and $(2)\rightarrow (3)$ also remains open.

\begin{openqs}\
 \begin{enumerate}[(i)]
  \item Is there an AEC $\K$ that fails JEP eventually below $\kappa$, $\kappa$ being the first measurable or the first super 
compact, but satisfies JEP in all cardinalities above $\kappa$?
\item Same question as (i), but satisfy JEP in \emph{one} cardinality above $\kappa$?
 \end{enumerate}
\end{openqs}

Although in this paper we will not deal with it, we ought to say that similar considerations about categoricity  have long 
occupied researchers in the area, e.g. the main open problem for AECs (Shelah's Categoricity Conjecture) is whether there exists some $\kappa$ such that $(1)$ implies $(2)$ for categoricity\footnote{Categoricity means that $P(\lambda)$ is the property that ``there exists a unique model (up to isomorphism) of size $\lambda$''.}.

A lower bound for $\mu_{AP}(\kappa)$ is given in \cite{KLH}. For every cardinal $\kappa$ 
and every $\alpha$ with $\kappa\le\alpha<\kappaplus$ there exists an AEC 
$W_\alpha$ in a vocabulary of size $\kappa$, $W_\alpha$ has  countable L\"{o}wenheim-Skolem number,  and it satisfies 
amalgamation up 
to $\beth_{\alpha}$, but fails amalgamation in $\beth_{\kappaplus}$. This proves that $\mu_{AP}(\kappa)\ge\beth_{\kappaplus}$.

On the other hand, results from \cite{BaldwinBoneyHanfNumbers}, prove that if $LS(\K)=\kappa$, then $\mu_{AP}(\kappa)$ is  no larger than the first strongly compact above $\kappa$. 
Combining the results form \cite{KLH} and \cite{BaldwinBoneyHanfNumbers}, $\mu_{AP}(\kappa)$, the Hanf number for AP, is contained in 
the interval between $\beth_{\kappaplus}$ and the first strongly compact above $\kappa$. As noted in 
\cite{BaldwinBoneyHanfNumbers}, 
the gap between these two cardinals is immense. 

The picture is not very promising for JEP either. From \cite{BaldwinBoneyHanfNumbers}, we know that the first strongly compact above $\kappa$ is also an upper bound for $\mu_{JEP}(\kappa)$. 
A lower bound is given by the examples from \cite{JEP} and \cite{Complete}. 
From Fact 3.37 in \cite{JEP}, for every countable $\alpha$ there exists an AEC $(\K_\alpha,\subset)$ defined by a 
universal $\lomegaone$-sentence such that $\LS(\K_\alpha)=\aleph_0$ and $\K_\alpha$ satisfies JEP up to and including 
$\beth_\alpha$, but it has no larger models. This proves that $\bethomegaone$ is a lower bound for $\mu_{JEP}(\aleph_0)$, the Hanf 
number for JEP. 

The above example can easily be generalized: for every $\kappa\le\alpha<\kappaplus$, there exists an AEC $(\K_\alpha,\subset)$ 
defined by a universal $\mathcal{L}_{\kappaplus,\omega}$-sentence such that $\LS(\K_\alpha)=\kappa$ and $\K_\alpha$ 
satisfies JEP up to and including $\beth_\alpha$, but it has no larger models. This proves that as in the case for AP, $\mu_{JEP}(\kappa)$ is contained in the interval between $\beth_{\kappaplus}$ and the first strongly compact above $\kappa$.

Stimulated by early versions of \cite{Complete}, Baldwin and Shelah announced in \cite{BaldwinShelahHanfNumbers} that (under 
certain set theoretic hypotheses) there exists a \emph{complete} $\lomegaone$-sentence with maximal models in cofinally many 
cardinalities below the first measurable. They also note that every model of size equal to or larger than the first 
measurable has a proper $\lomegaone$-elementary extension; one can take a countably complete ultrapower. This proves the first 
measurable cardinal to be equal to $\mu_P(\aleph_0)$, where $P$ is the property of ``maximality''. More precisely, ``maximality'' is defined to be the property $P(\lambda)$: the exists a maximal model of size $\lambda$. 

Notice that this introduces yet another notion of a Hanf number. If $\mu$ is the first measurable cardinal, then no 
$\lomegaone$-sentence has a maximal model above $\mu$. In the language of Fact \ref{hanfme},  letting $P$ be 
``maximality'' and $\kappa$ be the first measurable, (1) and (2) are equivalent for the trivial reason that they are both always false. The example from 
\cite{BaldwinShelahHanfNumbers} proves that (3) does not imply (1) (or (2)) for maximality. So, this notion of a Hanf number does 
not fit into Definition \ref{hanfp}, which is our working definition for this paper. 

The main result of this paper is to prove a lower bound for the Hanf number for JEP. In particular, we show that $\mu_{JEP}(\aleph_0)$ 
is bounded below by the first measurable cardinal. Since by results of \cite[Theorem 3.1]{MagidorIdentity}, the first measurable 
and the first 
strongly compact can be consistently equal, this proves that the known bounds are consistently optimal. Of course, the first 
measurable and the first strongly compact can also be different. So, this leaves open the question whether the lower bound 
from the current paper or the upper bound from \cite{BaldwinBoneyHanfNumbers} can be improved.

Our example was inspired by \cite{BaldwinShelahHanfNumbers}. In fact, the idea of the main construction in the current paper 
appears as Example 3.0.3. in \cite{BaldwinShelahHanfNumbers}. Our contribution is to compute the JEP- and AP- spectra of this 
example. Notice that our $\lomegaone$-sentence (as well as the sentence from Example 3.0.3 in \cite{BaldwinShelahHanfNumbers}) 
is incomplete. This is in contrast to the main construction of \cite{BaldwinShelahHanfNumbers}, where Baldwin and Shelah go 
in great lengths to create a \emph{complete} sentence. It is an open question whether the results from this paper can extend to 
the complete example of \cite{BaldwinShelahHanfNumbers}. 

In Section \ref{vocabulary} we describe the construction. The  JEP-spectrum is given in Section \ref{JEP} and the AP-spectrum in 
Section \ref{apsection}.

For the main definitions of Abstract Elementary Classes, we refer the reader to Baldwin \cite{BaldwinCategoricity}; however, they 
are not necessary to understand the main construction and theorems.

\section{The main example}\label{vocabulary}

Fix an infinite cardinal $\chi$, although this can be taken to be $\aleph_0$.  We will define a language $\tau^\chi$ and an 
AEC $\K^\chi$.  

The prototypical elements of $\K^\chi$ will be structures of the form 
$$M = \left(\kappa, \cP(\kappa), {}^\chi \cP(\kappa), \in, \vee,\wedge,\cdot^c,\one, \cap, \pi_\alpha\right)_{\alpha<\chi}$$
where
\begin{enumerate}
	\item ${}^\chi \cP(\kappa)$ are the $\chi$-length sequences from $\cP(\kappa)$ (the power set of $\kappa$);
	\item $\in$ is the (extensional) `member of' relation between $\kappa$ and $\cP(\kappa)$;
	\item $\left( \cP(\kappa),\vee,\wedge,\cdot^c,\one\right)$ is the standard Boolean Algebra;
	\item $\cap:{}^\chi \cP(\kappa) \to \cP(\kappa)$ returns the intersection of all elements of the sequence; and
	\item $\pi_\alpha:{}^\chi \cP(\kappa) \to \cP(\kappa)$ returns the $\alpha$th element of the sequence.
\end{enumerate}

However, we don't want the entire theory of these structures and only demand that the members of $\K^\chi$ satisfy a particular 
subset of their $\cL_{\chi^+, \omega}$-theory.

Formally, set $\tau^\chi$ to consist of
\begin{itemize}
	\item sorts $K$, $P$, and $Q$;
	\item constant $\one$;
	\item unary functions $\cdot^c$, $\cap$, and $\pi_\alpha$ for $\alpha < \chi$;
	\item binary functions $\wedge$ and $\vee$; and
	\item a binary relation $\e$.
\end{itemize}

Now we define a sentence $\psi^\chi \in \cL_{\chi^+, \omega}$ that consists of the conjunction of the following first-order 
assertions
\begin{enumerate}
	\item $\e \subset K \times P$ is extensional;
	\item $\left( P,\vee,\wedge,\cdot^c,\one\right)$ is a Boolean Algebra;
	\item $\e$ interacts with the Boolean Algebra operations in the expected way: $\one$ contains every element of $K$, a 
complement contains exactly the elements not in the original, etc;
	\item $P$ contains all singletons: $\forall x \in K \exists Y \in P \forall z \in K\left(z \in Y \iff z = x\right)$; and
\end{enumerate}
the conjunction of the following infinitary assertions:
\begin{enumerate}
	\item[(5)] the functions $\pi_\alpha:Q \to P$ are jointly extensional in the sense that
	$$\forall A, B \in Q \left( A = B \iff \bigwedge_{\alpha < \chi}\pi_\alpha(A) = \pi_\alpha(B)\right)$$
	\item[(6)] the function $\cap:Q \to P$ returns the intersection in the sense that
	$$\forall A \in Q \forall x \in K \left( x \e \cap A \iff \bigwedge_{\alpha<\chi} x \e \pi_\alpha(A)\right)$$
\end{enumerate}

Given any model $M$ of $\psi^\chi$, there is a natural embedding of the Boolean Algebra induced on $P^M$ into the Boolean Algebra 
on $\cP(K^M)$, the powerset of $K^M$.

\begin{definition}
Let $M \subset N$ model $\psi^\chi$ and $X \in P^M$.
\begin{enumerate}
	\item Define $\widehat{X}^M := \{x\in K^M\mid M \vDash x\e X\}$.
	\item  \emph{$M$ and $N$ agree on $X$} iff $\widehat{X}^M=\widehat{X}^N$. 
	\item  \emph{$M$ and $N$ agree on finite subsets} iff they agree on every $X\in P^M$ such that $\widehat{X}^M$ is a 
finite 
subset of $K^M$. 
	\item  Define $\K^\chi$ to be the collection of all models of $\psi^\chi$ of size $\ge\chi$. For $M,N\in \K^\chi$, let 
$M\subm^\chi N$ if $M\subset N$ and $M$ and $N$ 
agree on finite subsets. \\
	Following standard convention, we will often use $\K^\chi$ to refer to the pair $(\K^\chi, \subm^\chi)$.
\end{enumerate}
\end{definition}

Crucial to our later analysis, the lack of full elementarity in the $\subm^\chi$ relation means that $M$ and $N$ do not need to 
agree on all sets.  In particular, $M$ and $N$ agree on $\one$ iff $K^M = K^N$.  Note also that $\widehat{X}^M = \widehat{X}^N 
\cap K^M$.

\begin{proposition}\label{aec}
For each $\chi$, $(\K^\chi,\subm^\chi)$ is  an Abstract Elementary Class with L\"{o}wenheim-Skolem number $\chi$. 
\end{proposition}
\begin{proof} This is an easy argument and we leave the details to the reader. One approach is that a definitional expansion (by 
adding ``finite subset'' functions) turns $\K^\chi$ into a universal class.
\end{proof}

\begin{observation}\label{bounds} \
\begin{enumerate}
	\item It  is immediate from the definition that the size of $P$ is bounded by $2^{|K|}$ and the size of $Q$ by 
$|P|^{\chi}$. There is no restriction on the size of $K$. Therefore,  $\psi^\chi$ has models in all infinite cardinalities.
	\item \label{canon}Given any model $M \vDash \psi^\chi$, we can find an isomorphic copy  $\widehat{M}$ of $M$ such that 
	\begin{enumerate}
		\item $K^{\widehat{M}} = K^M$;
		\item $P^{\widehat{M}} \subset \cP(K^M)$; and
		\item $Q^{\widehat{M}} \subset {}^\chi \cP(K^M)$.
	\end{enumerate}
	However, this choice is very non-canonical.
\end{enumerate}
\end{observation}

\begin{definition} \
\begin{enumerate}[(a)]
 \item An AEC $\K$ satisfies the \emph{$\kappa$-Joint Embedding Property} or JEP($\kappa)$ iff $\K_\kappa$ is not empty and for 
all $M_0, M_1 \in \K_\kappa$, there are $N \in \K$ and $\K$-embeddings $f_\ell:M_\ell \to N$. The \emph{joint embedding spectrum 
of $\K$} or JEP-spectrum is the collection of all cardinals $\kappa$ such that $\K$ satisfies JEP($\kappa)$.
\item  Similarly, define AP($\kappa)$ and the AP-spectrum for the amalgamation property.
\item A model $M$ of $\psi^\chi$ is of type $(\lambda,\kappa)$ with $\lambda\ge\kappa$, if $|M|=\lambda$ and 
$|K^M|=\kappa$. 
\item Given $M \in \K^\chi$, we say that an ultrafilter $U$ on the Boolean Algebra $P^M$ is \emph{$Q^M$-complete} iff for every 
$A\in Q^M$, if $\pi_\alpha^M(A) \in U$ for all $\alpha <\chi$, then $\cap^M(A)\in U$. 
\end{enumerate}
\end{definition}
%

\section{The Joint Embedding Property}\label{JEP}
In this section we determine the joint embedding spectrum of $\K^\chi$. It turns out (Lemma \ref{jointembed}) that the question 
of 
when two models can be jointly embedded depends on whether they can be extended in a certain way, called $K$-extendibility.

\begin{definition}\label{ext-def}
Let $M \in \K^\chi$.
\begin{enumerate}
	\item $M$ is \emph{$K$-extendible} iff there is $N \in \K^\chi$ such that $M \subm^\chi N$ and $K^M \subsetneq K^N$.
	\item $M$ is \emph{$K$-maximal} iff whenever $N \in \K^\chi$ has $M\subm^\chi N$, then $K^M \subsetneq K^N$.
\end{enumerate}
\end{definition}

So given a $K$-maximal model, it is either $K$-extendible or truly maximal.  The next lemma connects $K$-extendibility to the 
existence of sufficiently complete ultrafilters.

\begin{lemma}\label{extension} 
Let $M\in\K^\chi$.  The following are equivalent:
\begin{enumerate}
	\item $M$ is $K$-extendible.
	\item For all cardinals $\lambda$, there exists some $N\in \K^\chi$ with $M\subm^\chi N$ and $|K^N\setminus 
K^M|=\lambda$ .
	\item There exists a non-principal ultrafilter $U$ on $P^M$ that is $Q^M$-complete.
\end{enumerate}
\end{lemma}
\begin{proof} The implication $(2)\Rightarrow(1)$ is immediate. We start by proving that $(1)\Rightarrow(3)$.
Suppose we have $M, N \in \K^\chi$ that witness the $K$-extendibility of $M$. By assumption, there exists some  $d 
\in K^N\setminus K^M$. We then build 
an ultrafilter $U_d\subset P^M$ as follows: for $X \in P^M$, we set

$$X \in U_d \text{ if and only if } N \vDash d \e X$$

Then $U_d$ is an ultrafilter by the standard set-theoretic argument; see, e.g. Theorem 5.6 in 
\cite{KanamoriHigherInfinite}.\footnote{Another approach is to observe that $U_d$ is the set of all sets whose $\mu$-measure 
equals to $1$, where $\mu$ is defined on $P^M$ by $\mu(X)=1$ when $N\vDash d\in X$, and $\mu(X)=0$ otherwise. $U_d$ is an 
ultrafilter because $\mu$ is $2$-valued.} We only prove $U_d$ is non-principal and $Q^M$-complete. 

Assume that $U_d$ is principal and generated by some $X_0\in P^M$. By definition, $P^M$ contains all finite subsets of 
$K^M$.\footnote{This means that if $K_0$ is a finite subset of $K^M$, there exists some $X\in P^M$ with $\widehat{X}^M=K_0$. See 
also 
clause (4) of $\psi^\chi$.} It 
follows that $\widehat{X_0}^M$ is a singleton, say $\widehat{X_0}^M=\{x\}$. By definition of $\subm $, $M,N$ agree on all finite 
subsets of $K^M$. In 
particular, $\widehat{X_0}^M=\widehat{X_0}^N$. By the definition of $U_d$, this implies that $d=x\in 
K^M$. Contradiction. 

Let $A\in Q^M$ such that $X_\alpha :=\pi^M_\alpha(A)\in U_d$, for all $\alpha < \chi$.  Since $N \vDash d \e X_\alpha$ for every 
$\alpha < \chi$, we must have $N \vDash d \e \cap A$.  Thus, $\cap^N(A) = \cap^M(A) \in U_d$.
\newline

Next we prove that $(3)\Rightarrow(2)$. Suppose that $M \in \K^\chi$ and there is a non-principal, $Q^M$-complete ultrafilter 
$U$ on $P^M$. Without loss of generality $M$ satisfies the 
conclusions of Observation \ref{bounds}.(\ref{canon}).  Define $N$ to be an $\prec^\chi_{\K}$-extension of $M$ as follows: 
Extend $K^M$ by $\lambda$-many new elements to form $K^N$. Say $K^N = K^M \cup \lambda$. Define $\e^N$ on the new elements: For all $x\in K^N\setminus K^M$ and all $X \in P^M$, stipulate that 
$$N \vDash x \e X \iff X \in U.$$ 
This implies that for every $X\in P^M$, either $\widehat{X}^N=\widehat{X}^M$, or $\widehat{X}^N=\widehat{X}^M\cup (K^N\setminus 
K^M)$. Which of the two is the 
case is determined by membership in $U$.

Additionally, we extend $P$ so that $P^N$ is the Boolean Algebra that is generated by $P^M$ together with all finite subsets of $K^N$. In particular, for every $F$ finite 
subset of $K^N$ which is not a subset of $K^M$, we introduce a new element $X_F$ in $P^N$ and define $\e^N$ is such a way that 
$\widehat{X_F}^N=F$.

Moreover, $M$ and $N$ agree on $Q, \pi_\alpha$ and $\cap$, and on $\e\upharpoonright_{K^M\times P^M}$, and $\vee,\wedge,^c$ 
have the intended interpretations. 

We verify that $N$ is in $\K^\chi$ and $M\subm^\chi N$. $\e^N$ is extensional because the same is true for $\e^M$ and by definition of $N$. Condition (2) is immediate and condition (3) is easy 
to verify using the fact that $U$ is an ultrafilter. We leave the details to the reader. (4) and (5) are immediate. 

We only prove (6) using the $Q$-completeness of $U$. If $A\in Q^N=Q^M$, then $Y=\cap^N A$ is in $U$ if and 
only if for all $\alpha$, $\pi_\alpha(A)\in U$. The one direction follows from $U$ being an ultrafilter, the other direction 
follows from $U$ being $Q$-complete. 

We take cases: If $Y=\cap^N A\in U$, then $\widehat{Y}^N=\widehat{Y}^M\cup (K^N\setminus 
K^M)$ and $\widehat{\pi_\alpha(A)}^N=\widehat{\pi_\alpha(A)}^M\cup (K^N\setminus K^M)$, for all $\alpha$. If $Y\notin U$, then 
$\widehat{Y}^N=\widehat{Y}^M$ and $\widehat{\pi_\alpha(A)}^N=\widehat{\pi_\alpha(A)}^M$, for \emph{some} $\alpha$ (maybe not for 
all $\alpha$). In 
either case, $\widehat{Y}^N=\cap_\alpha \widehat{\pi_\alpha(A)}^N$, which completes the proof that $N\in \K^\chi$. 

It remains to prove $M\subm^\chi N$. The fact that $M\subset N$ is immediate. Let $X$ be in $P^M$ so that $\widehat{X}^M$ is a 
finite set. We need to prove that $\widehat{X}^N=\widehat{X}^M$. Since $\widehat{X}^M$ is finite and $U$ is non-principal, 
$X\notin U$. Therefore, $\widehat{X}^N=\widehat{X}^M$ as desired.
\end{proof}

\begin{corollary}\label{extension2}
 Assume $M\prec^\chi_{\K} N$. If $M$ is not $K$-extendible, then the same is true for $N$. 
\end{corollary}
\begin{proof}
 By definition of $K$-extendibility, $K^N$ must equal $K^M$. Since $M\subset N$, $P^M$ is a sub-Boolean Algebra of $P^N$. If $N$ 
were $K$-extendible, by Lemma \ref{extension}, there would be a non-principal ultrafilter $U$ on $P^N$ which is $Q^N$-complete. 
The restriction of $U$ on $P^M$ leads to a contradiction. 
\end{proof}

The key factor for determining whether two $M_0,M_1\in \K^\chi$ can be jointly embedded is the size of 
$K^{M_0}$ and $K^{M_1}$. If $|K^{M_0}|=|K^{M_1}|$, then this is possible as seen by the next Lemma \ref{universal}. In fact more 
is true: For every $\kappa$, there is a single structure $M\in \K^\chi$ that can embed all other $M'\in\K^\chi$ with 
$|K^{M'}|=\kappa$. 
This universality property holds true because the Boolean Algebra $(P,\vee,\wedge,^c,\one,\zero)$ interpreted in 
$M$ (the universal model) coincides with the Boolean Algebra of the power set of $K^M$. 
If $|K^{M_0}|<|K^{M_1}|$, then $M_0,M_1$ can be jointly embedded if and only if $M_0$ is $K$-extendible. This is the content of 
Lemma \ref{jointembed}

\begin{lemma}\label{universal} 
Let $\kappa\geq \chi$. There exists a model $M\in \K^\chi$ of type $(2^\kappa,\kappa)$ such that for any other $N\in\K^\chi$ with 
$|K^{N}|=\kappa$, there is an $\subm^\chi$-embedding from $N$ to $M$. Moreover, $M$ is $K$-maximal.
\end{lemma}
\begin{proof}  Let $M$ be the standard model generated by $K^M=\kappa$, $P^M=\cP(\kappa)$, $Q^M={}^\chi\cP(\kappa)$. We claim 
that 
$M$ is the desired model.

Let $N \in \K^\chi$ with $|K^N|=\kappa$.  Define a $\K^\chi$-embedding $f$ from $N$ to $M$ as follows: $f$ is a bijection from 
$K^{N}$ to $\kappa$; this is possible by cardinality assumptions. For each $X\in P^{N}$, set $f(X) = \{f(x) \mid x \in 
\widehat{X}^{N}\} \in \cP(\kappa)$.  For each $A\in Q^{N}$, let $f(A)$ be the sequence $\left\langle 
f\left(\pi_\alpha^N(A)\right)\mid \alpha < \chi\right\rangle \in {}^\chi \cP(\kappa)$.  It is immediate that $f$ is an 
$\K^\chi$-embedding. 

For the moreover, the extensionality of $\in$ and the $\pi_\alpha$ imply that any extension of $M$ cannot grow $P$ or $Q$ without 
growing $K$.
\end{proof}

\begin{lemma}\label{jointembed} Let $M_0,M_1\in\K^\chi$. If $|K^{M_0}|<|K^{M_1}|$, then $M_0$ and $M_1$ can be jointly embedded 
if and only if $M_0$ is $K$-extendible. 
\end{lemma}

\begin{proof} Left-to-right: Suppose that $M_0$ and $M_1$ can be jointly embedded into some $M\in\K^\chi$. Then 
$|K^M|\ge|K^{M_1}|>|K^{M_0}|$. By Definition \ref{ext-def}, $M_0$ is $K$-extendible.

Right-to-left: Assume that $M_0$ is $K$-extendible. By Lemma \ref{extension}, there exists some $M_0'\in\K^\chi$, $M_0\subm^\chi 
M_0'$ and $|K^{M_0'}|=|K^{M_1}|$. Use Lemma \ref{universal} to joint embed $M_0'$ and $M_1$ to a common $M$. Then $M$ serves also 
as the joint embedding of $M_0$ and $M_1$. 
\end{proof}

Recall that there are no countably complete, non-principal ultrafilters on any set of size less than the first measurable 
cardinal.  More generally, fixing $\chi$, there is no $\chi^+$-complete, non-principal ultrafilter on any set of size less than 
the first measurable larger than $\chi$ (if one exists).  If there is no measurable above $\chi$, then there are no 
$\chi^+$-complete, non-principal ultrafilters at all. We utilize these facts to prove the existence of models that are not 
$K$-extendible.

\begin{definition}
Given a cardinal $\kappa$, set 
$$m(\kappa):= \inf \{\lambda \mid \lambda > \kappa \text{ and }\lambda\text{ is measurable}\}$$
If there are no measurable cardinals above $\kappa$, then set $m(\kappa) = \infty$ (which is greater than every ordinal by 
convention).
\end{definition}

\begin{lemma}\label{maximal} 
Let $\chi \leq \kappa < m(\chi)$. There exists a model $M \in \K^\chi$ of type $(2^\kappa,\kappa)$ that is $\subm^\chi$-maximal. 
In particular, $M$ is not $K$-extendible. 
\end{lemma}
\begin{proof} Let $M$ be the standard model on $\left(\kappa, \cP(\kappa), {}^\chi \cP(\kappa)\right)$ as in Lemma 
\ref{universal}.  We know it is $K$-maximal, and we claim that $M$ is maximal.   By the remark following Definition 
\ref{ext-def}, 
it suffices to show that it is not $K$-extendible.  

If it were $K$-extendible, then Lemma \ref{extension} would imply that there is a $Q^M$-complete, non-principal ultrafilter $U$ 
on the Boolean Algebra $P^M$.  However, since $P^M = \cP(\kappa)$, $U$ is an ultrafilter on $\kappa$.  Moreover, since $Q^M = 
{}^\chi \cP(\kappa)$, the $Q^M$-completeness of $U$ is a different name for $\chi^+$-completeness (in the normal sense).  This 
would imply that there is a $\chi^+$-complete ultrafilter on $\kappa$; however, this contradicts $\kappa < m(\chi)$.
\end{proof}

\begin{theorem}\label{failure} 
Let $\chi \leq \kappa<m(\chi)$. Then $\K^\chi$ fails $JEP(2^\kappa)$.
\end{theorem}
\begin{proof} Let $M$ be the model from Lemma \ref{maximal}. $M$ is of type $(2^\kappa,\kappa)$. Let $N$ be any model in $\K$ of 
type $(2^\kappa,2^\kappa)$. Take for instance $K^N=2^\kappa$, $P^N$ contains all finite and co-finite subsets of $K^N$ and $Q^N$ 
is empty. 

By Lemma \ref{jointembed}, $M$ and $N$ can be jointly embedded if and only if $M$ is $K$-extendible. But $M$ is not 
$K$-extendible 
by Lemma \ref{maximal}, which proves the theorem.  
\end{proof}

Once we are above a measurable, $K$-extendibility and, therefore, joint embedding become trivial to accomplish.

\begin{lemma}\label{extendible}
If $\kappa \geq m(\chi)$, then every model in $\K^\chi_\kappa$ is $K$-extendible.
\end{lemma}

\begin{proof}
Let $U$ be a $m(\chi)$-complete, non-principal ultrafilter and $M \in \K^\chi_\kappa$.  Since measurable cardinals are strong 
limits, $|K^M| \geq m(\chi)$.  By \L o\'{s}' Theorem for infinitary logics, the ultrapower $\prod M/U$ is a $\cL_{m(\chi), 
m(\chi)}$-elementary extension of $M$ (up to isomorphism).  In particular, $M \subm^\chi \prod M/U \in \K^\chi$.  Moreover, 
$K^{\prod M/U} = \prod K^M/U \supsetneq K^M$.  This witnesses that $M$ is $K$-extendible.
\end{proof}

\begin{theorem}\label{abovemeasurable} 
If $\kappa\ge m(\chi)$, then $\K^\chi$ satisfies $JEP(\kappa)$. 
\end{theorem}
\begin{proof} As we noted, if $M\in \K^\chi_\kappa$, then $|K^M|\ge m(\chi)$. The statement follows 
from Lemma \ref{universal} in the case $|K^{M_0}| = |K^{M_1}|$, or otherwise from Lemmas \ref{jointembed} and \ref{extendible}. 
\end{proof}

We can also show that joint embedding holds at strong limit cardinals.

\begin{theorem}\label{stronglimit} 
If $\kappa \geq \chi$ is a strong limit cardinal, then $\K^\chi$ satisfies $JEP(\kappa)$.
\end{theorem}
\begin{proof}
The strong limit assumption implies that if $M\in \K^\chi_\kappa$, then $|K^M|=\kappa$ by Observation \ref{bounds}.  To prove 
$JEP(\kappa)$, let $M_0,M_1\in\K^\chi$ both of size $\kappa$. It follows that $|K^{M_0}|=|K^{M_1}|=\kappa$, so they are both 
jointly embeddable into the universal model $N$ given by Lemma 
\ref{universal}.  Since $\K^\chi$ is an AEC (Proposition \ref{aec}), there is $N'\subm^\chi N$ of size $\kappa$ that contains the 
images of 
$M_0,M_1$. This proves $JEP(\kappa$). 
\end{proof}
Notice that Theorem \ref{stronglimit} holds true even for $\kappa=\chi$. 

Suppose that there is a measurable above $\chi$ and let $\mu =m(\chi)$.  Then Theorem \ref{stronglimit} yields a cofinal sequence 
in $\mu$ on which joint embedding holds, while Theorem \ref{failure} yields a cofinal sequence on which JEP fails. Under GCH this 
gives a complete 
characterization of the JEP-spectrum of $\K$.

\begin{corollary} Assume GCH holds. $\K^\chi$ satisfies $JEP(\kappa$) if and only if 
$\kappa\geq \chi$ is a limit cardinal below $m(\chi)$ or $\kappa\ge m(\chi)$.
\end{corollary}

\begin{corollary} 
The Hanf number $\mu_{JEP}(\aleph_0)$ is at least a measurable cardinal. 
\end{corollary}

In the rest of this section we determine the joint embedding spectrum even when GCH fails, in particular for cardinals 
$\chi\le \kappa <m(\chi)$ such that $\kappa < 2^{<\kappa}$.  Our results shows that $JEP(\kappa)$ fails 
for almost all such cardinals. Depending on cardinal arithmetic there might be some cardinals for which our method does not give an answer. The precise statement is given in Theorem \ref{JEPspectrum}. 

Our abstract tool is the following lemma, which reduces our problem to finding a particular model.

\begin{lemma}\label{failure2} 
Suppose that $\kappa \geq \chi$ and $\lambda < \kappa \le 2^\lambda$.  If $\K^\chi$ contains a model of type $(\kappa, \lambda)$ 
that is not $K$-extendible, then $\K^\chi$ fails $JEP(\kappa)$. 
\end{lemma}
\begin{proof} Let $M$ be a model as in the assumption and let $N$ be a model of type $(\kappa,\kappa)$. Then $\|M\|=\|N\|=\kappa$ 
and $|K^N|>|K^M|$.  By Lemma \ref{jointembed}, $M$ and $N$ cannot be jointly embedded, so $JEP(\kappa)$ fails.
\end{proof}

When $\kappa = \chi$, such a model is easy to build.

\begin{lemma}\label{nonextendible0} 
There exists a model $M\in \K^{\chi}$ of type $(\chi, \chi)$ which is not $K$-extendible. 
\end{lemma}

\begin{proof} 
Define $M$ by setting $K^M=\chi$; $P^M$ is the Boolean Algebra that is generated by all finite subsets of $\chi$ plus all tails of the form $[\alpha,\chi)$, with $\alpha <\chi$; and $Q^M$ contains for every limit 
ordinal $\beta\le\chi$ the sequence of tails $\left\langle [\alpha,\chi)\mid \alpha < \beta\right\rangle$.  

A couple of remarks are at hand before we proceed:
\begin{enumerate}
 \item We do not impose any completeness requirements on  $P^M$. This means that the size of $P^M$ is equal to $\chi$. 
 \item If $\beta$ is a limit ordinal less than $\chi$, then the sequence of tails in $Q^M$, $\left\langle [\alpha,\chi)\mid \alpha < \beta\right\rangle$, has length smaller than $\chi$. Since sequences in $Q^M$ must be of length $\chi$, we tacitly assume that some of the sets in this sequence repeat. 
\end{enumerate}

It 
follows that $M$ has type $(\chi, \chi)$ and $M\in\K^{\chi}$. 

We also claim that $M$ is not $K$--extendible, or, equivalently by Lemma \ref{extension}, there is no non-principal 
$Q^M$-complete ultrafilter on $P^M$. Assume otherwise and let $U$ be such an ultrafilter. We prove by induction on $\alpha$ 
that $[\alpha,\chi)\in U$. For the successor stage use the fact that $U$ is non-principal. For the limit stages use the fact 
that $U$ is $Q^M$-complete. 

Therefore, $U$ contains all the sets of the form $[\alpha,\chi)$, $\alpha < \chi$. Using $Q^M$-completeness once more, $U$  
must 
contain $\cap_{\alpha} [\alpha,\chi)=\emptyset$. Contradiction. 
\end{proof}

\begin{corollary}\label{failJEP0}
 $\K^\chi$ fails $JEP(\kappa)$ for all $\chi^+\leq \kappa \leq 2^{\chi}$. 
\end{corollary}
\begin{proof}
 Let $\chi^+\leq \kappa \leq 2^{\chi}$ and let $M$ be the model from Lemma \ref{nonextendible0}. Define some $N\in \K^\chi$  that extends $M$ and has size $\kappa$: $K^N=K^M$, $P^N$ is a Boolean Algebra that extends $P^M$ and has size $\kappa$, and $Q^N=Q^M$. If $N$ were $K$-extendible, then by Lemma \ref{extension}, there would be a non-principal $Q^N$-complete ultrafilter on $P^N$. The restriction of $U$ on $P^M$ contradicts Lemma \ref{nonextendible0}. 
\end{proof}

Our next goal is to extend Lemma \ref{nonextendible0} and Corollary \ref{failJEP0} to higher cardinalities. This will be 
achieved in Corollary \ref{failJEP1}. We need some preliminary work before we can prove Corollary \ref{failJEP1}. 

The proof of the following fact is standard, see, e. g., \cite[Lemma 4.2.3]{ChangKeislerModelTheory}.

\begin{fact}\label{partition}
 Let $B$ be a Boolean Algebra on $\kappa$ and let $U$ be an ultrafilter on $B$. Then $U$ is $\lambda$-complete, for some 
$\lambda\le\kappa$, if and only if for every $W\subset B$, a partition of $B$ of size $|W|<\lambda$, there exists some $w\in W$ 
that belongs to $U$.
\end{fact}

Recall that a cardinal $\kappa$ is weakly compact if and only if for every $\kappa$-complete Boolean Algebra 
$B\subset P(\kappa)$ generated by $\kappa$-many subsets, there is a $\kappa$-complete non-principal ultrafilter on $B$. We will 
consider a weakening of this large cardinal notion as in \cite[Definition 2.1]{BoneyUnger}.

\begin{definition}
 A cardinal $\kappa$ is \emph{$\delta$-weakly compact} for $\delta\le\kappa$ iff every $\kappa$-complete Boolean Algebra 
$B\subset 
P(\kappa)$ generated by $\kappa$-many subsets has a $\delta$-complete non-principal ultrafilter on $B$. 
\end{definition}

Note that $\kappa$-weakly compact is the same as weakly compact. 

\begin{lemma}\label{weaklycompact}
Fix an infinite cardinal $\chi$. If $\chi < \kappa \leq m(\chi)$ and $\kappa$ is $\chi^+$-weakly 
compact, then $\kappa$ is weakly compact. 
\end{lemma}
\begin{proof}
 Assume otherwise. That is there exists some $\kappa$-complete \comment{$m(\chi)$-complete } Boolean Algebra $B\subset P(\kappa)$ 
generated by $\kappa$-many subsets such that there is no $\kappa$-complete non-principal ultrafilter on $B$, but there is a 
$\chi^+$-complete, non-principal ultrafilter $U$ on $B$.  Let $\mu^+$ be the least cardinal so $U$ is not $\mu^+$-complete.  Note 
that we must have $\chi^+ \leq \mu < \kappa$.

Since $U$ is not $\mu^+$-complete, by Fact \ref{partition}, there exists some partition $W=\{w_i\mid i<\mu\}$ of $B$ such that 
$w_i\notin U$ for all $i < \mu$. Define a function $f:\kappa\rightarrow \mu$ by 
\[f(x)=i \text{ if $x\in w_i$}\]

The function $f$ is defined on $\kappa$ and is onto $\mu$. Use $f$ to define a complete Boolean Algebra $C\subset 
\cP(\mu)$ and some ultrafilter $V$ on $C$ as follows: $Y\in C$ if and only if $f^{-1}(Y)\in B$, and $Y\in V$ if and only if 
$f^{-1}(Y)\in U$. 

It is routine to verify that $C = \cP(\mu)$ and that $V$ is a $\mu$-complete, non-principal ultrafilter on $C$.  Thus, $\mu$ must 
be measurable, contradicting the assumption that $\chi <\mu < \kappa \leq m(\chi)$.
\end{proof}

We introduce two functions on cardinals, one standard and one not.
\begin{definition}\begin{enumerate} 
    \item For any cardinal $\kappa$, 
\begin{itemize}
 \item $\beth_0(\kappa)=\kappa$
 \item $\beth_{\alpha+1}(\kappa)=2^{\beth_\alpha(\kappa)}$
 \item $\beth_\lambda(\kappa)=\sup_{\alpha<\lambda} \beth_\alpha(\kappa)$, for $\lambda$ limit 
\end{itemize}
If no $\kappa$ is mentioned, we assume that $\kappa=\aleph_0$. 
    \item  For any cardinal $\kappa$,
    $$A(\kappa) = \begin{cases} \kappa^{<\kappa} & \text{if $\kappa$ is regular}\\
    2^\kappa & \text{if $\kappa$ is singular}\end{cases}$$
\end{enumerate}
\end{definition}

\begin{fact}\label{bethfact}\
 \begin{enumerate}[(a)]
  \item All strong limit cardinals are of the form $\beth_\lambda$ for some limit $\lambda$ and, given any $\kappa$, the least strong limit above $\kappa$ is $\beth_\omega(\kappa)$.
  \item If $B$ is a $\kappa$-complete Boolean algebra on $\cP(\kappa)$ that contains the finite subsets of $\kappa$, then $B$ must be of size $A(\kappa)$.  Recall that we assume the Boolean algebras we deal with have all finite subsets of $\kappa$.
 \end{enumerate}

\end{fact}

\begin{corollary}\label{failJEP1} Assume $\chi \leq \kappa<m(\chi)$. 
If $\kappa$ is not weakly compact, then $JEP(\lambda)$ fails for all $\lambda$ satisfying
$$\max\{\kappa^{+},A(\kappa)\}\le\lambda\le 
2^{\kappa}$$
\end{corollary}

Note that $2^\kappa$ always satisfies this inequality, but it can be the \emph{only} cardinal there if case $\kappa$ is singular or GCH holds at $\kappa$.

\begin{proof}
 By Lemma \ref{weaklycompact}, $\kappa$ is not $\chi^+$-weakly compact. By definition there is a Boolean Algebra $B$ on 
$\kappa$ that admits no non-principal $\chi^+$-complete ultrafilter.  The size of $B$ is $A(\kappa)$ by Fact \ref{bethfact}. 
Then use $B$ to construct a model $M\in \K^\chi$ of type $\left(A(\kappa),\kappa\right)$ that is not $K$-extendible by defining 
$K^M=\kappa$, $P^M=B$ and $Q^M=B^\chi$. By Lemma \ref{failure2}, if $\kappa<A(\kappa)$, $\K^\chi$ 
fails $JEP(A(\kappa))$.

For $\lambda$ any cardinal in the interval $(A(\kappa), 2^\kappa]$ (if any), we work similarly. Construct a model 
$N\in \K^\chi$ of type $(\lambda,\kappa)$ such that $M\prec^\chi_{\K} N$. Define $N$ by letting 
$K^N=K^M=\kappa$, $P^N$ is a Boolean Algebra extension of $P^M=B$ that has size $\lambda$ and $Q^N=Q^M=B^\chi$. By Corollary 
\ref{extension2}, $N$ is not $K$-extendible and by Lemma \ref{failure2} again, $\K^\chi$ fails $JEP(\lambda)$.
\end{proof}

We can strengthen this failure by using induction to build many failures between $\kappa$ and the first strong limit above $\kappa$.

\begin{lemma}\label{failJEP2}
Let $\chi \leq \kappa < m(\chi)$ such that $\kappa$ is not weakly compact.  Then $JEP(\lambda)$ fails for all $\lambda$ satisfying 
$$\max\{\kappa^+, A(\kappa)\} \leq \lambda < \beth_\omega(\kappa)$$
except perhaps when 
    there is a singular limit $\mu$ such that $2^{<\mu} \le \lambda < 2^\mu$, in which case our method does not determine whether JEP$(\lambda)$ holds or not.
\end{lemma}

 The cardinal restrictions on $\lambda$ may seem strange, but are necessary from our methods.  Essentially, we will apply Corollary \ref{failJEP1} to an interval of cardinals to get failures of JEP.  When applying to successor $\lambda$, the  start of the new interval covers the first cardinal missed by the previous intervals. However, at limit cardinals, there can be a gap at the places indicated in the lemma.  Notice that under GCH this exception can not happen.  Additionally, the JEP is guaranteed to fail on the $\beth_n(\kappa)$'s.

\begin{proof}
We apply Corollary \ref{failJEP1} to all cardinals $\lambda \in \left[\kappa, \beth_\omega(\kappa)\right)$ to get failure of JEP on each interval 
$$I_\lambda = [\max\{\lambda^+, A(\lambda)\}, 2^\lambda].$$

We take the following cases:

\begin{itemize}
\item  \underline{$\lambda = \lambda_0^+$ successor}:
The start of the interval $I_\lambda$ is at worst $\left(2^{\lambda_0}\right)^+$, which is the first cardinal above $I_{\lambda_0}$. Using that $\lambda$ is regular and $A(\lambda) = \lambda^{<\lambda}$, we have that
$$\max\{\lambda^+, A(\lambda)\} \leq (2^{<\lambda})^+.$$

\item \underline{$\lambda$ limit cardinal:} Set $I_{<\lambda} = \cup_{\mu<\lambda} I_\mu$. The first cardinal not in  $I_{<\lambda}$ is either $2^{<\lambda}$ or $\left(2^{<\lambda}\right)^+$, depending on whether $I_{<\lambda}$ has supremum or a maximum respectively. We split this case into the following subcases. 

    \item \underline{$\lambda$ is regular and $2^{<\lambda} = 2^\mu$ for some $\mu < \lambda$.}\\
    In this case, the first cardinal above $I_{<\lambda}$ is $(2^{<\lambda})^+$.  By regularity, $A(\lambda) = \lambda^{<\lambda}$. Then there is no gap between $I_{<\lambda}$ and $I_\lambda$ because
    $$\max\{\lambda^+, \lambda^{<\lambda}\} \leq (2^{<\lambda})^+$$
    \item \underline{$\lambda$ is regular and $2^{<\lambda} > 2^\mu$ for all $\mu < \lambda$.}\\
    In this case, the first cardinal above $I_{<\lambda}$ is $2^{<\lambda}$. By regularity, $A(\lambda) = \lambda^{<\lambda} = 2^{<\lambda}$, so the left endpoint of $I_\lambda = \max\{\lambda^+, 2^{<\lambda}\}$.  If $2^{<\lambda} = \lambda$, then $\lambda$ would have been a strong limit. But we noted already that the first strong limit above $\kappa$ is $\beth_\omega(\kappa)>\lambda$. So,  $2^{<\lambda} > \lambda$ and there is no gap between $I_{<\lambda}$ and $I_\lambda$. 
    \item \underline{$\lambda$ is singular and $2^{<\lambda} = 2^\mu$ for some $\mu < \lambda$.}\\
    In this case, the continuum function is eventually constant below $\lambda$ and $2^\lambda=2^{<\lambda}=2^\mu$. By singularity, $A(\lambda) = 2^{\lambda}$. So $I_\lambda$ begins at (and only contains) $2^\lambda$ and $I_\lambda$ is a subset of $I_{<\lambda}$. 
    \item \underline{$\lambda$ is singular and $2^{<\lambda} > 2^\mu$ for all $\mu < \lambda$.}\\
    In this case, the first cardinal above $I_{<\lambda}$ is $2^{<\lambda}$.  By singularity, $A(\lambda) = 2^{\lambda}$. So $I_\lambda$ begins at (and only contains) $2^\lambda$, thus the gap.
    
\end{itemize}

Putting the above items together, we have the above lemma statement.
\end{proof}


Summarizing the results of this section, we have shown that the JEP spectrum of $\K^\chi$ is as desired.

\begin{theorem}\label{JEPspectrum}
Fix an infinite cardinal $\chi$.  Then there is an AEC $\K^\chi$ with $LS(\K^\chi) = \chi$ (that is given by models of a sentence in $\mathbb{L}_{\chi^+,\omega}$) whose JEP spectrum satisfies the following:
\begin{enumerate}
    \item JEP holds cofinally below $m(\chi)$;
    \item JEP fails cofinally below $m(\chi)$; and
    \item JEP holds everywhere above $m(\chi)$.
\end{enumerate}
In particular, we have
\begin{itemize}
    \item JEP holds at every strong limit and above $m(\chi)$;
    \item if $\chi \leq \kappa < m(\chi)$, then JEP fails at $2^\kappa$; and
    \item more generally, for regular $\kappa$ with $\chi \leq \kappa < m(\chi)$, JEP fails on the interval
    $$\left[ \max\{\kappa^+, \kappa^{<\kappa}\}, 2^\kappa\right]$$
\end{itemize}

\end{theorem}
\begin{proof}
The first bullet point is Theorems \ref{stronglimit} and \ref{abovemeasurable}.  The last two bullets are Theorem \ref{failure} and Lemma \ref{failJEP2}.
\end{proof}

\begin{corollary}
If GCH holds, then JEP fails in $\K^\chi$ exactly at the $\kappa$ satisfying $\chi \leq \kappa < m(\chi)$ that are not strong limit.
\end{corollary}

%
%

We finish this section by producing models of $ZFC +$ ``there is a measurable'' where Theorem \ref{JEPspectrum} 
characterizes the $JEP$ spectrum of $K^\chi$. The main tool is the following theorem of Paris and Kunen (see \cite[Theorem 
21.3]{JechsSetTheory}). 

\begin{fact}\label{ParisKunen}
Assume GCH and let $\kappa$ be a measurable cardinal. Let $D$ be a normal measure on $\kappa$ and let $A$ be a set of regular
cardinals below $\kappa$ such that $A\notin D$. Let $F$ be a function on $A$ such that
$F (\alpha) < \kappa$ for all $\alpha \in A$, and:
\begin{enumerate}[(i)]
 \item $cf F (\alpha) > \alpha$;
 \item $F (\alpha_1 ) \le F(\alpha_2)$ whenever $\alpha_1 \le \alpha_2$.
\end{enumerate}

Then there is a generic extension $V [G]$ of  $V$ with the same cardinals and
cofinalities, such that $\kappa$ is measurable in $V [G]$, and for every $\alpha\in  A$, 
$V [G] \vDash  2^\alpha = F(\alpha)$.

Moreover, the powersets of cardinals not in $A$ have the smallest possible cardinality that satisfies
$\kappa<cf(2^\kappa)$ and that the powerset function is increasing.
\end{fact}

\begin{theorem}\label{club} Assume GCH and fix an infinite cardinal $\chi$.  Given a club $C$ on $m(\chi)$, there is a generic 
extension $V[G]$ that preserves cardinalities and cofinalities, $m(\chi)$ remains a measurable cardinal and 
$\K^\chi$ satisfies $JEP(\lambda)$ iff $\lambda \in \lim C$ or $\lambda \geq m(\chi)$.
\end{theorem}

\begin{proof} The goal is to force the cardinal arithmetic of $V[G]$ to make the limit points of $C$ the uncountable strong limit 
cardinals while preserving the measurability of $m(\chi)$.  Since limits of strong limit cardinals are also strong limits, it 
suffices that we ensure that all cardinals in $\lim C\setminus \lim\lim C$ are strong limits.

Let $U$ be a normal ultrafilter on $m(\chi)$.  $U$ contains all clubs, so $\lim C \in U$.  Given $\lambda \in \lim C\setminus 
\lim\lim C$, $\lambda$ has cofinality $\omega$ and the set $\lim C\cap \lambda$ is bounded in $\lambda$. Let $\{\kappa^\lambda_n 
\mid n < \omega\}$ be an increasing sequence of regular cardinals converging to $\lambda$ and choose $\kappa_0^\lambda$ such 
that it is the least regular cardinal above all cardinals in $\lim C\cap \lambda$.  Then define a function $F$ with domain $\{\kappa^\lambda_n \mid n < 
\omega, \lambda \in \lim C - \lim \lim C\}$ by
$$F(\kappa^\lambda_n) = \kappa^\lambda_{n+1}$$

Using Theorem \ref{ParisKunen}, we can force to preserve cofinalities and the 
measurability of $m(\chi)$ while enforcing
$$V[G] \vDash ``2^{\kappa^\lambda_n} = \kappa^\lambda_{n+1} {}"$$

Thus we have guaranteed that in $V[G]$, $\beth_\omega(\kappa_0^\lambda)=\lambda$. In particular all cardinals in  $\lim C$ are strong limit cardinals. We prove that the reverse is true too. 

Let $\mu$ be a limit cardinal (in $V[G]$) that is not in $\lim C$. We prove that $\mu$ is not strong limit. Let $\lambda$ be the least cardinal in $\lim C$ above $\mu$. If $\kappa_0^\lambda\le \mu$, then we noted that  $\lambda=\beth_\omega(\kappa_0^\lambda)$ is the least strong limit above $\kappa_0^\lambda$. In this case, $\mu$ is not strong limit. 
If $\mu<\kappa^\lambda_0$, then $\mu$ must be a singular limit of cardinals in $\lim C$. Given that $C$ is a club, $\mu$ must also belong to $\lim C$, which is a contradiction. Therefore, $\mu$ can not be a strong limit cardinal. 

So, in $V[G]$, $\lim C= \{\beth_\lambda|\lambda \text{:limit and } \lambda< m(\chi)\}$. It 
follows from the moreover clause of Theorem \ref{ParisKunen} that in $V[G]$, 
$2^{\beth_\lambda}=(\beth_\lambda)^+$, for all limit $\lambda<m(\chi)$, that is GCH holds at strong limit cardinals. 
 
We claim that $JEP(\lambda)$ if and only if  $\lambda\in \lim C$ or $\lambda\ge m(\chi)$.  The right-to-left direction is from 
Theorems \ref{stronglimit} and \ref{abovemeasurable}, and the fact that the cardinals in $\lim C$ are strong limits. 

For the 
left-to-right direction we take cases. If $\kappa$ is a strong limit that is not a weakly compact cardinal, then we claim that JEP fails for all $\lambda$ in the interval $[\kappaplus,\beth_\omega(\kappa))$. We observed already that GCH hols at strong limits. So $2^\kappa=\kappaplus$ and $\max\{\kappaplus,A(\kappa)\}=\kappaplus$. By  Lemma \ref{failJEP2}, we know that JEP fails on the interval $[\max\{\kappaplus,A(\kappa)\},\beth_\omega(\kappa))=[\kappaplus,\beth_\omega(\kappa))$, except possibly some cardinals mentioned in Lemma \ref{failJEP2}. We argue that there are no such cardinals, i.e. there is no  singular limit $\mu$ such that $2^{<\mu}  < 2^\mu$ in the interval  $[\kappaplus,\beth_\omega(\kappa))$. Let $\lambda=\beth_\omega(\kappa)$. By the way $\kappa_0^\lambda$ was defined, it must be $\kappa_0^\lambda=\kappaplus$. Let $\mu\in [\kappaplus,\beth_\omega(\kappa))$ be a limit cardinal and let $n\in \omega$ be the least such that $\kappa_{n}^\lambda<\mu<\kappa_{n+1}^\lambda$. Since $\kappa^\lambda_{n+1}$ is a regular cardinal, by the moreover clause of Theorem \ref{ParisKunen},  in $V[G]$ it holds that $2^\mu=\kappa_{n+1}^\lambda$. In addition, in $V[G]$, $2^{\kappa_n^\lambda}\le 2^{<\mu}\le 2^{\mu}=\kappa_{n+1}^\lambda$. Therefore, $2^{<\mu}=2^\mu$ and there is no gap in the failures of JEP  given by Lemma \ref{failJEP2} in the interval $[\kappaplus,\beth_\omega(\kappa))$. 

If $\kappa$ is a weakly compact cardinal, then $\kappa$ is a strong limit and again $\kappaplus=2^\kappa$. By Theorem \ref{failure}, JEP fails at $\kappaplus$. Moreover, apply Lemma \ref{failJEP2} to $\kappaplus$. Since  $\max\{\kappaplusplus,\left(\kappaplus\right)^{\left(<\kappaplus\right)}\}=\kappaplusplus$,  JEP fails on the interval $[\kappaplusplus,\beth_\omega(\kappa))$. Combined together, JEP fails on the interval $[\kappaplus,\beth_\omega(\kappa))$.
\end{proof}

\section{Amalgamation}\label{apsection}
In this section, we investigate the amalgamation spectrum of $\K^\chi$ and show that amalgamation will always eventually fail, 
regardless of large cardinals. 

We start by providing a strong condition for when elements can be identified in the amalgam.  Then we prove Lemma \ref{AP}, which 
is an analogue of Lemma \ref{extension} for disjoint amalgamation. 

Recall that by the proof of Lemma \ref{extension}, if $M_0\subm^\chi M_1$ and $K^{M_1}\setminus K^{M_0}\neq \emptyset$, then for 
every $d\in K^{M_1}\setminus K^{M_0}$ we can define a $Q^{M_0}$-complete, non-principal ultrafilter $U_d$ on $P^{M_0}$ by 
$$X\in U_d \text{ if and only if }M_1\models d\e X$$
When there is ambiguity, we will refer to this ultrafilter by $U_d^{M_0, M_1}$.

\begin{lemma}\label{dapfail}
 Assume $M_0,M_1,M_2\in\K^\chi$  and let $d_i\in 
K^{M_i}\setminus K^{M_0}$, $i=1,2$. If there exists an amalgam  $N\in\K^\chi$ of $M_1,M_2$ over $M_0$ where $d_1$ and $d_2$ are identified, then 
$U_{d_1}^{M_1,M_0}=U_{d_2}^{M_2,M_0}$.
\end{lemma}
\begin{proof}
Suppose that $d_1$ and $d_2$ are identified in the amalgam $N$. Since $M_1,M_2$ are $\subm^\chi$-substructures of $N$, 
$U_{d_1}^{M_1,M_0}=U_{d_1}^{N,M_0}=U_{d_2}^{N,M_0}=U_{d_2}^{M_2,M_0}$.
%

\end{proof}

\begin{definition}
 Let $N, M_l\in \K^\chi$, $l=0,1,2$ and $N$ is an amalgam of  $M_1$ and $M_2$ over $M_0$. Say that the amalgamation is 
\emph{disjoint 
on the $K$-sort} if in $N$ no elements of $K^{M_1}\setminus K^{M_0}$ are identified with any elements of $K^{M_2}\setminus 
K^{M_0}$.
\end{definition}

\begin{lemma}\label{AP} Let $M_0,M_1,M_2\in\K^\chi$ with $M_0\subm^\chi M_1$ and $M_0\subm^\chi M_2$.
There is an amalgamation of $M_1,M_2$ over $M_0$ that is disjoint on the $K$-sort if and only if for every $d\in 
K^{M_i}\setminus K^{M_0}$  the ultrafilter $U_d^{M_i,M_0}$ on $P^{M_0}$ defined above can be 
extended to a non-principal ultrafilter on $P^{M_{3-i}}$ that is $Q^{M_{3-i}}$-complete, for $i=1,2$. 
\end{lemma}

Note that $U_d^{M_i,M_0}$ is an ultrafilter on the Boolean Algebra (defined 
by) $P^{M_0}$.  Since $M_0\subm^\chi M_{3-i}$, we have that $P^{M_0}$ is a sub-Boolean Algebra of $P^{M_{3-i}}$. Consequently,  
$U_d^{M_i,M_0}$ is a filter on $P^{M_{3-i}}$. Then Lemma \ref{AP} says that there is an amalgamation of $M_1,M_2$ over $M_0$ which is disjoint on the $K$-sort exactly when 
$U_d^{M_i,M_0}$ can be extended to a non-prinicipal and $Q^{M_{3-i}}$-complete ultrafilter on $P^{M_{3-i}}$ and this can be done for every $d$.

\begin{proof}[Proof of Lemma \ref{AP}] 

First, suppose that $N$ is a $K$-disjoint amalgam of $M_1,M_2$ over $M_0$. Consider the case where $d\in K^{M_1}\setminus 
K^{M_0}$. The case $d\in K^{M_2}\setminus K^{M_0}$ is symmetric and we omit it. 

In the amalgam $d$ is an element of $K^N\setminus K^{M_2}$. Define as 
before the ultrafilter $U_d^{N,M_2}$. Then $U_d^{N,M_2}$ is a non-principal ultrafilter on $P^{M_2}$ that is 
$Q^{M_2}$-complete. We need to prove that $U_d^{N,M_2}$ extends $U_d^{M_1,M_0}$. 
Let $X\in P^{M_0}$. Then 
\begin{align*}
 X\in U_d^{N,M_2} &\text{ iff } N\models d\e X\\
           &\text{ iff } M_1\models d\e X\\
           &\text{ iff } X\in U_d^{M_1,M_0}.
\end{align*}

Second, suppose that we have this extension property and let $V_d^{M_i,M_0}$ denote the $Q^{M_{3-i}}$-complete ultrafilter on 
$P^{M_{3-i}}$ that extends $U_d^{M_i,M_0}$. Furthermore, suppose that $M_1$ and $M_2$ are disjoint except for the common copy of 
$M_0$ and that the elements of $Q^{M_i}$ are actually $\chi$-sequences from $P^{M_i}$.

We define the amalgam $N$ of $M_1,M_2$ over $M_0$: $K^N$ equals $K^{M_1}\cup K^{M_2}$.

$P^N$ is the Boolean Algebra generated by $P^{M_1}\cup P^{M_2}$ and all the finite subsets of $K^N$. 
We identify two elements $X\in P^{M_1}$ and $Y\in P^{M_2}$, if $\widehat{X}^{M_1}$ and $\widehat{Y}^{M_2}$ is the same subset of 
$K^{M_0}$. 

$Q^N$ equals $Q^{M_1}\cup Q^{M_2}$, 
modulo the requirement that if $A\in Q^{M_1}$ and $B\in Q^{M_2}$ are such that for all $\alpha<\chi$, 
$\pi_\alpha^{M_1}(A)=\pi_\alpha^{M_2}(B)$, then we identify $A$ and $B$ in the amalgam. 

All that remains is to define $\e^N$.  It suffices to define $\e^N$ on $K^N \times (P^{M_1} \cup P^{M_2})$, and then extend it to 
the rest of $P^N$ by the Boolean Algebra rules.  We require that $\e^N$ extend $\e^{M_1} \cup \e^{M_2}$.  Suppose $d \in 
K^{M_i}\setminus K^{M_0}$ and $X \in P^{M_{3-i}}$. Then we set 
$$N \vDash d \in X \iff X \in V_d^{M_i, M_0}$$

Notice that if $X\in P^{M_0}$, then $X\in V_x^{M_i,M_0}$ iff $X\in U_x^{M_i,M_0}$ 
iff $M_i\models x\e X$. So, $N$ and $M_i$ agree on $\e$ on their common domain. 

The reader can verify that $N \in \K^\chi$ and that $M_i \subm^\chi N$ working as in the proof of Lemma 
\ref{extension}.  In particular, the $Q^{M_{3-i}}$-completeness of 
$V_d^{M_i, M_0}$ is crucial as in Lemma \ref{extension}.
\end{proof}

Observe that the proof of the above Lemma does not yield a disjoint amalgam for $M_1, M_2$. The reason is that some elements of 
$P^{M_1},P^{M_2}$ and some elements of $Q^{M_1},Q^{M_2}$ may be identified. Nevertheless, amalgamation is disjoint on the 
$K$-sort.

Using Lemma \ref{AP}, we prove that $\K^\chi$ fails amalgamation above $2^{\chi^+}$. The idea of the proof is due to Spencer 
Unger.

\begin{theorem}\label{negative}

Let $\kappa \geq 2^{\chi^+}$. Then $\K^\chi$ fails $AP(\kappa)$.  
\end{theorem}

\begin{proof} First, we will build a filter $F$ on $\cP(\kappa)$ generated by $\leq\kappa$-many sets that cannot be extended to a 
$\chi^+$-complete ultrafilter on all of $\cP(\kappa)$; indeed, we will identify a Boolean Algebra $P_1$ and collection of 
$\chi$-sequences $Q_1$ such that $F$ cannot be extended to a $Q_1$-complete ultrafilter on $P_1$.  To do so, partition $\kappa$ 
into $\{A_\alpha \mid \alpha < \chi^+\}$ and define the filter $F$ on $\kappa$ by, for $X \subset \kappa$,

$$X \in F \text{ if and only if there is } \beta < \chi^+\text{ such that } \bigcup_{\alpha > \beta} A_\alpha \subset X$$

Note that $F$ is $\chi^+$-complete and contains every cofinite set (and even the co-$\chi$-sized sets).  Let $P_0 \subset 
\cP(\kappa)$ be the Boolean Algebra generated by the sets measured by $F$. 
Then set $P_1 \supset P_0$ be the Boolean Algebra generated by

$$\left\{\bigcup_{\alpha \in S} A_\alpha \mid S \subset \chi^+\right\}$$

Now, define our set of $\chi$-sequences by

$$Q_1:= \left\{\left\langle \bigcup_{\alpha \in S_i} A_\alpha \mid i < \chi\right \rangle \mid \left\langle S_i 
\right\rangle_{i<\chi} \in {}^\chi \cP(\chi^+) \right\}$$

{\bf Claim:} $F$ cannot be extended to a $Q_1$-complete filter $G$ that measures all sets in $P_1$.\\

Suppose it could.  From $G$, we can define a non-principal ultrafilter $U$ on $\chi^+$ by, for $Y \subset \chi^+$,
$$Y \in U \text{ if and only if } \bigcup_{\alpha \in Y} A_\alpha \in G$$
Each of these sets is in $P_1$ by construction, so this is a non-principal ultrafilter following the standard argument.  
Moreover, $U$ is $\chi^+$-complete precisely because $G$ is $Q_1$-complete.  This is a contradiction because there can be no 
$\chi^+$-complete, non-principal ultrafilter over $\chi^+$.  This completes the proof of the claim.

Second, we reverse engineer the proof of Lemma \ref{AP} to show that this construction forces a failure of amalgamation.  We 
build a triple of models $M_0,M_1,M_2$. Unless otherwise specified $\e$ is the regular $\in$ (`belongs to') relation and the 
Boolean Algebra operations are the usual intersection, union and complement. We specify $(K, P, Q)$:
\begin{itemize}
	\item $M_0$ is defined by $(\kappa, P_0, \emptyset)$;
	\item $M_1$ is defined by $(\kappa \cup \{d\}, P_0, \emptyset)$; $d$ belongs to some $X\in P_0$ if and only if $X\in 
F$; and
	\item $M_2$ is defined by $(\kappa, P_1, Q_1)$.
\end{itemize}
Note that $F$ is an ultrafilter on $P_0$ as $P_0$ contains precisely the sets that $F$ measures.  By construction, $M_0 
\subm^\chi M_1, M_2$.  Tracing the definition, $U_{d}^{M_1, M_0}=F$. So by Lemma \ref{AP}, the triple 
$(M_0,M_1,M_2)$ can be amalgamated iff $F$ can be extended to a $Q^{M_2}$-complete filter on $P^{M_2}$.  However, this is 
impossible by the claim.  

We finish the proof by observing that all these models have size 
$\kappa+\left(2^{\chi^+}\right)^\chi = \kappa$. So $\K^\chi_\kappa$ fails AP$(\kappa)$.
\end{proof}

\section*{Acknowledgements} The authors would like to thank John Baldwin for his comments on an early version of this paper, and the anonymous referee for improving the paper and catching an error in a previous version of Lemma \ref{failJEP2}. The 
second author would like to thank the Aristotle University of Thessaloniki. This paper was written while he was a visitor at 
the Aristotle University.

 \bibliographystyle{plain}   

\bibliography{AllBibliography}  
\end{document}